\documentclass[10pt]{article}

\usepackage{amsfonts}
\usepackage{amsmath}
\usepackage{amsthm}
\usepackage{amssymb}
\usepackage{fullpage}

\newtheorem{theorem}{Theorem}

\newtheorem{prop}[theorem]{Proposition}

\newtheorem{example}{Example}

\newtheorem*{claim*}{Claim}

\def\A{\mathcal{A}}
\def\X{\mathcal{X}}
\def\0{{\bf 0}}
\def\1{{\bf 1}}

\title{Equilibria in a Hypercube Spatial Voting Model}

\author{A. Nicholas Day\thanks{E-mail: a.nick.day@gmail.com}
\and
J. Robert Johnson\thanks{School of Mathematical Sciences, Queen Mary University of London, London E1 4NS, UK.  E-mail: r.johnson@qmul.ac.uk.}}

\begin{document}

\maketitle


\begin{abstract}
We give conditions for equilibria in the following Voronoi game on the discrete hypercube. Two players position themselves in $\{0,1\}^d$ and each receives payoff equal to the measure (under some probability distribution) of their Voronoi cell (the set of all points which are closer to them than to the other player). This game can be thought of as a discrete analogue of the Hotelling--Downs spatial voting model in which the political spectrum is determined by $d$ binary issues rather than a continuous interval.

We observe that if an equilibrium does exist then it must involve the two players co-locating at the majority point (ie the point representing majority opinion on each separate issue). Our main result is that a sufficient condition for an equilibrium is that on each issue the majority option is held by at least $\frac{3}{4}$ of voters. The value $\frac{3}{4}$ can be improved slightly in a way that depends on $d$ and with this improvement the result is best possible. We give similar sufficient conditions for the existence of a local equilibrium.

We also analyse the situation where the distribution is a mix of two product measures. We show that either there is an equilibrium or the best response to the majority point is its antipode.
\end{abstract}

\section{Introduction}

The aim of this paper is to analyse some aspects of a combinatorial game on the hypercube motivated by the Hotelling--Downs spatial voting framework.

Spatial voting concerns competition among candidates for votes in some geometric arena representing a political spectrum. In the classical Hotelling-Downs model \cite{hotelling} \cite{downs}, the political spectrum is a continuous interval. This means that a candidate's or voter's opinion is represented by a single real number (which in practice could be though of as a left-right axis). Suppose that we have some distribution (probability measure) of voters on $[0,1]$ and two players (candidates) position themselves in $[0,1]$ each with the aim of maximising the measure of the set of points which are closer to them than their opponent. An early observation of Hotelling \cite{hotelling} anticipating the Median Voter Theorem, is that the only equilibrium position (ie position where neither player has an incentive to move) in this game is for the two players to co-locate at the median point of the measure (the value $x$ for which $\mu([0,x])=\mu([x,1])=\frac{1}{2}$).

Another equally natural choice for the political spectrum is to represent a candidate's position by recording their view on $d$ binary (yes/no) issues. This moves the setting from continuous geometry to combinatorics and the discrete hypercube $Q_d=\{0,1\}^d$ equipped with the Hamming distance. A point of $Q_d$ is interpreted as an opinion on $d$ separate binary issues where on each issue the options are $0$ or $1$. The space $Q_d$ now represents all possible opinion patterns (the political spectrum), with distance between two points being the number of issues on which they differ. As in the continuous case, we have a distribution of voters and assume that each voter will vote for the candidate which is closest to them, that is the candidates who agrees with them on the largest number of issues. Note that if a voter is equidistant from the two candidates we will think of them as sharing their vote equally between each candidate. This model appears in \cite{MMW} and although very natural, only preliminary results on it exist.

Let us give a small example to illustrate the model. Suppose that $d=3$ and that there are $5$ voters with positions $(0,0,0)$, $(0,0,0)$, $(0,1,1)$, $(1,0,1)$ and $(1,1,0)$ (we could equally well think of each of each voter as representing $\frac{1}{5}$ of a large society). If we were to hold a vote on each of the $3$ separate issues then option $0$ wins (by a margin of $3$ to $2$) in each case. However, a candidate positioned at $(0,0,0)$ will lose to a candidate positioned at $(1,1,1)$ (because $3$ of the $5$ voters are closer to $(1,1,1)$ than to $(0,0,0)$ in Hamming distance). Moreover, the position with candidates at $(0,0,0)$ and $(1,1,1)$ is not an equilibrium because the $(0,0,0)$ player can improve their payoff by moving to any other position apart from $(1,1,1)$. In fact this particular example has no equilibrium.

This simple example shows that some voter distributions do not have an equilibrium. Our aim is to give conditions on the voter distribution which do guarantee the existence of an equilibrium. Such results can be thought of as conditions under which this discrete hypercube model behaves analogously to the classical $1$-dimensional Median Voter model.

In all our results, the probability measure we put on the hypercube will be critical. Indeed, in the unweighted cube (uniform distribution of voters) however the two candidates are positioned, the payoff to each is $\frac{1}{2}$ (as observed in \cite{MMW}). The easiest way to see this is to note that for any points $A,B\in Q_d$ there is an automorphism of $Q_d$ which exchanges $A$ and $B$. 

Section 2 contains a formal description of the game and establishes notation. 

Section 3 concerns general conditions for the existence of equilibria. It is easy to check that if an equilibrium does exist then it must involve the two players co-locating at the majority point (ie the point representing majority opinion on each separate issue). Our first main result (Theorem \ref{hypercube-mvt1}) is that a sufficient condition for an equilibrium to exist is that on each issue the majority option is held by at least $\frac{3}{4}$ of voters. The $\frac{3}{4}$ can be decreased by a small amount depending on $d$ and we also show with this improvement the result is sharp (Theorem \ref{non-equilib}). We also extend this result to conditions involving opinions on small sets of issues. 

Following \cite{MMW} we say that a configuration is a $k$-local equilibrium if neither candidate can improve their vote share by moving up to distance $k$ from their current position. For instance, in the $d=3$ example above co-locating at $(0,0,0)$ is a $2$-local equilibrium. We will show that the approach which leads to our existence results can also be used to give conditions for local equilibria.

The existence of a majority point assumes that there is no issue on which voters are split perfectly evenly. Section 4 exhibits some examples showing that a range of behaviour is possible in this degenerate case. 

In Section 5 we analyse the case when the voter distribution is given by a mix of product measures. This could be thought of a society in which voters are split into two groups. Each voter's opinions are random and independent, but voters in the first group tend to prefer $0$ on all issues while voters in the second group tend to prefer $1$ on all issues. We show that in this case there is a dichotomy with either co-locating at the majority point being an equilibrium or polarisation occurring where the best response to the majority point is its antipode.

We finish in Section 6 with some suggestions for further work.

In addition to \cite{MMW} in which this model first appears, we mention some other relevant work. In contrast to the continuous setting, there is relatively little work on spatial competition games on graphs. These include \cite{DT,TDU,GMPPR,MMPS}. The closest to our problem is \cite{FMM} which discusses the existence of equilibria for $k$-players in a class of transitive graphs which includes the hypercube, but in the setting where the underlying distribution is uniform. 

Finally, we mention that the model introduced in \cite{MMW} includes the restriction that each candidate may only choose a position within a fixed subset (different for each candidate) of $Q_d$ locations. This models the idea that in reality there may be some positions a candidate is unwilling to take (perhaps because of ideology or to maintain trust). We look at the unrestricted game which is mathematically natural and may still give some social insight.

\section{Set-up and Notation}

We are interested in the hypercube instance of the following game described on a general metric space. 

Let $\X$ be some metric space with a probability measure $\mu$ defined on it. We think of $\X$ as the \emph{political spectrum}, that is the set of all possible opinions. For $S\subseteq\X$, the measure $\mu(S)$ will express the proportion of members of a society who hold an opinion in $S$.

Suppose now that $2$ \emph{players} position themselves in $\X$. The \emph{Voronoi cell} of a player is the set of all points in $\X$ which are closer to that player than to the other player. For $A,B\in \X$, we write $V(A,B)$ for the Voronoi cell of a player who chooses $A$ against an opponent who chooses $B$, and $T(A,B)$ for the set of points which are equidistant between $A$ and $B$. That is
\begin{align*}
V(A,B)&=\{X\in \X : d(X,A)<d(X,B)\}\\
T(A,B)&=\{X\in \X : d(X,A)=d(X,B)\}.
\end{align*}
Notice that $V(A,B), V(B,A), T(A,B)$ partition $\X$.

The position of the game in which Player 1 chooses point $A$ and Player 2 chooses point $B$ will be represented by the pair $(A,B)$. We write $P_1(A,B)$ and $P_2(A,B)$ for the payoffs to Player 1 and Player 2 respectively in the position $(A,B)$ and define these by:
\[
P_1(A,B)=\mu(V(A,B))+\frac{1}{2}\mu(T(A,B)); \qquad P_2(A,B)=\mu(V(B,A))+\frac{1}{2}\mu(T(A,B))
\]
Note that $P_1(A,B)+P_2(A,B)=1$.

We say that a pair of points $(A,B)$ with $A,B\in \X$ is an equilibrium if each player's position is a best response to the other. More formally,
\begin{itemize}
\item $P_1(A,B)\geqslant P_1(A',B)$ for all $A'\in \X$ (that is $A$ is a \emph{best response to $B$})
\item $P_2(A,B)\geqslant P_2(A,B')$ for all $B'\in \X$ (that is $B$ is a \emph{best response to $A$})
\end{itemize}
in other words, neither player can strictly improve their payoff by moving to a different point of $\X$.  

For the remainder of this paper, our political spectrum $\X$ will be the discrete hypercube $Q_d=\{0,1\}^d$ equipped with the Hamming metric. 

If $X\in Q_d$ we write $X$ as the $d$-tuple $(x_1,x_2,x_3,\ldots, x_d)\in\{0,1\}^d$. We write $|X|$ for the number of $1$'s in this $d$-tuple, and $X^c$ for the complementary $d$-tuple $(1-x_1,1-x_2,\ldots, 1-x_d)$. The $k$th \emph{layer} of $Q_d$ is $\{X\in Q_d : |X|=k\}$. For $X,Y\in Q_d$ we write $d(X,Y)$ for the Hamming distance between $X$ and $Y$, that is
\[
d(X,Y)=|\{ i : x_i\not=y_i\}.
\]
We sometimes identify elements of $Q_d$ with subsets of $[d]=\{1,2,\ldots,d\}$ in the obvious way. In particular we write $[k]$ for the element $(\underbrace{1,\ldots,1}_{k},\underbrace{0,\ldots, 0}_{d-k})$ of $Q_d$. We also write $X\subseteq Y$ if $x_i\leqslant y_i$ for all $i$. We write $\0=(0,\ldots,0)$ and $\1=(1,\ldots,1)$.
 
Since $Q_d$ is finite, our probability measure consists of a function $\mu:Q_d\rightarrow\mathbb{R}_{\geqslant 0}$ giving each vertex a weight which satisfies $\sum_{V\in Q_d} \mu(V)=1$. This is extended to subsets of $Q_d$ by summing over vertices so $\mu(\A)=\sum_{V\in \A} \mu(V)$ for all $\A\subseteq Q_d$.

\section{Conditions for Equilibria}

We write $w_i^0$ for the total weight of the co-dimension $1$ subcube given by setting the $i$th coordinate to be $0$ (ie the proportion of voters choosing $0$ in a referendum on the $i$th issue).
\[
w_i^0=\mu(\{X\in Q_d: x_i=0\}).
\]
  
Suppose now that $w_i^0\not=\frac{1}{2}$ for all $i$. We define the \emph{majority point} $M$ by
\[
M_i=\left\{\begin{array}{ll}
0 & \text{ if } w_i^0>\frac{1}{2}\\
1 & \text{ if } w_i^0<\frac{1}{2}\\
\end{array}\right.
\]
So the majority point represents society's choice if a separate referendum was held on each of the $d$ issues.

Relabelling $0$ and $1$ in each coordinate if necessary we may assume that $w_i^0>\frac{1}{2}$ for all $i$ and so $M=\0$.
 
The assumption that $w_i^0\not=\frac{1}{2}$ is largely for tidiness; if we drop it then we may have a ``majority subcube'' whose free coordinates correspond to the $i$ for which $w_i^0=\frac{1}{2}$. Most of the results and observations that follow will still hold once suitably reformulated. We discuss the differences in Section 4.

We are interested in conditions for the existence of equilibria (defined above). If there is an equilibrium then it must be to co-locate at the majority point.

\begin{prop}\label{weighted-equilibrium}
Let $\mu$ be any probability distribution on $Q_d$ with $w_i^0>\frac{1}{2}$ for all $i$. If $(A,B)$ is an equilibrium then $A=B=\0$.
\end{prop}

\begin{proof}
Suppose that $(A,B)$ is an equilibrium. If $P_1(A,B)<\frac{1}{2}$ then $P_1(B,B)=\frac{1}{2}>P_1(A,B)$ (ie the first player improves their payoff by co-locating with the second player at $B$) and so $(A,B)$ is not an equilibrium. Similarly, if $P_1(A,B)>\frac{1}{2}$ then $P_2(A,B)<\frac{1}{2}$. Now the second player improves their payoff by co-locating with the first player at $A$ and so $(A,B)$ is not an equilibrium. We conclude that we must have $P_1(A,B)=\frac{1}{2}$. If $B\not=\0$ then we have some $i$ with $B_i=1$. Let $A'$ be identical to $B$ except on coordinate $i$ where $a_i'=0$. We have $P_1(A',B)=w_i^0>\frac{1}{2}=P_1(A,B)$ and so $(A,B)$ is not an equilibrium. It follows that we must have $A=B=\0$.
\end{proof}

A consequence of this is that we have an equilibrium if and only if $\0$ is a best response to $\0$.

Our main aim is this section is to determine some conditions on $\mu$ which guarantee the existence of an equilibrium.

Our first result shows that if all the $w_i^0$ are large enough then we do have an equilibrium, which necessarily consist of the two players co-locating at the majority point.

\begin{theorem}\label{hypercube-mvt1}
If $\mu$ is a distribution on $Q_d$ which satisfies
\[
w_i^0\geqslant\begin{cases}
\frac{3}{4}-\frac{1}{4d} & \text{ if $d$ is odd}\\
\frac{3}{4}-\frac{1}{4(d-1)} & \text{ if $d$ is even}
\end{cases}
\]
for all $i\in[d]$ then the position $(\0,\0)$ is an equilibrium.
\end{theorem}

\begin{proof}[Proof of Theorem \ref{hypercube-mvt1}]
Suppose that $(\0,\0)$ is not an equilibrium. Let $X\in Q_d\setminus\{\0\}$ be any point with $P_1(X,\0)> \frac{1}{2}$ (such a point must exist if $(\0,\0)$ is not an equilibrium). Without loss of generality $X=[k]$ for some $1\leqslant k\leqslant d$. 

Consider the sum
\[
S=\sum_{i=1}^k w_i^0
\]
Now, let us consider the contribution to $S$ of each $V\in Q_d$. If $V\in V(\0,X)$ then $V$ may contribute $\mu(V)$ to each of the $k$ summands in $S$ making a total of $k\mu(V)$ (this occurs in the case that $V=\0)$. If $V\in V(X,\0)$ then $|\{i : 1\leqslant i\leqslant k, v_i=0\}|<k/2$ and so $V$ contributes $\mu(V)$ to at most $\lfloor\frac{k-1}{2}\rfloor$ summands in $S$ for a total of $\lfloor\frac{k-1}{2}\rfloor \mu(V)$. Finally, if $k$ is even and $V\in T(X,\0)$ then $|\{i : 1\leqslant i\leqslant k, v_i=0\}|=k/2$ and $V$ contributes exactly $\frac{k}{2}\mu(V)$ to $S$.

Suppose that $k$ is odd. Letting $\mu(V(\0,X))=x_1$ and aggregating the bounds on the contributions to $S$ we get:
\begin{align*}
S&\leqslant k w(V(\0,B))+\frac{k-1}{2}\mu(V(B,\0)\\
&=kx_1+\frac{k-1}{2}\left(1-x_1\right)\\
&=\frac{k-1}{2}+\frac{k+1}{2}x_1\\
&<\frac{3k-1}{4}
\end{align*}
since $x_1=P_1(\0,X)<\frac{1}{2}$.

It follows by averaging, there must be some $i\in[k]$ with
\begin{equation}\label{wi-bound-odd}
w_i^0<\frac{1}{k}\left(\frac{3k-1}{4}\right)=\frac{3}{4}-\frac{1}{4k}.
\end{equation}

Similarly, if $k$ is even we let $\mu(V(\0,X))=x_1$ and $\mu(T(\0,X))=x_2$. Then
\begin{align*}
S&\leqslant k \mu(V(\0,X))+\frac{k}{2}\mu(T(\0,X))+\frac{k-2}{2}\mu(V(X,\0)\\
&=kx_1+\frac{k}{2}x_2+\frac{k-2}{2}\left(1-x_1-x_2\right)\\
&\leqslant kx_1+\frac{k}{2}x_2+\frac{k-2}{2}\left(1-x_1-\frac{x_2}{2}\right)\\
&=\frac{k-2}{2}+\frac{k+2}{2}\left(x_1+\frac{x_2}{2}\right)\\
&\leqslant\frac{3k-2}{4}
\end{align*}
since $x_1+\frac{x_2}{2}=P_1(\0,X)<\frac{1}{2}$.

It follows by averaging, there must be some $i\in[k]$ with
\begin{equation}\label{wi-bound-even}
w_i^0<\frac{1}{k}\left(\frac{3k-2}{4}\right)=\frac{3}{4}-\frac{1}{2k}.
\end{equation}

It follows that if $(\0,\0)$ is not an equilibrium, there must be some $k$ for which the corresponding inequality (\ref{wi-bound-odd}) or (\ref{wi-bound-even}) (depending on whether $k$ is even or odd) holds. Conversely, if (\ref{wi-bound-odd}) fails for all odd $k$ and (\ref{wi-bound-even}) fails for all even $k$ then $(\0,\0)$ is an equilibrium.

The righthand sides of (\ref{wi-bound-odd}) and (\ref{wi-bound-even}) are decreasing in $k$ so for all of them to fail it is enough that $w_i^0\geqslant\frac{3}{4}-\frac{1}{4d}$ when $d$ is odd and $w_i^0\geqslant\frac{3}{4}-\frac{1}{2d}$ when $d$ is even (in the later case the $k=d-1$ bound is stronger than the $k=d$ bound).

\end{proof}

Notice, also that the bound in the argument above really depends on $k$ (the distance from the majority point to the response that beats it) rather than $d$.

This allows us to use the argument to give conditions for local equilibria. As defined in \cite{MMW}, the position $(A,B)$ is called a \emph{$k$-local equilibrium} if: 
\begin{itemize}
\item $P_1(A,B)\geqslant P_1(A',B)$ for all $A'\in Q_d$ with $d(A,A')\leqslant k$
\item $P_2(A,B)\geqslant P_2(A,B')$ for all $B'\in Q_d$ with $d(B,B')\leqslant k$
\end{itemize}
in other words, neither player can strictly improve their payoff by moving to a different point of $Q_d$ within distance $k$ of their current position. This is a natural concept both mathematically and socially; a candidate may be reluctant to move to an extremely different position because of ideological reasons or to avoid losing trust.

By symmetry, $(\0,\0)$ is a $k$-local equilibrium if $P_1(X,\0)<\frac{1}{2}$ for all $X$ with $|X|\leqslant k$. The proof of Theorem \ref{hypercube-mvt1} immediately  establishes the following condition for $(\0,\0)$ to be a $k$-local equilibrium.

\begin{theorem}\label{local-mvt}
If $\mu$ is a probability measure on $Q_d$ with $w_i^0\geqslant\frac{3}{4}-\frac{1}{4k}$ for all $i$ then $(\0,\0)$ is a $k$-local equilibrium.
\end{theorem} 

The argument of Theorem \ref{hypercube-mvt1} can also be applied to weights of more complicated sets based on co-dimension $t$ subcubes through the majority point in place of $w_i^0$. Let $I$ be a $t$-element subset of $[n]$ and define
\[
w_I^m=w(\{X\in Q_d: x_i=0 \text{ for all but at most $m$ coordinates }i\in I\}).
\]
In other words, we look at a small subset $I$ of issues and take the proportion of all voters who prefer $0$ on at least some number of these issues. 

\begin{theorem}\label{hypercube-mvt2}
Suppose that $0\leqslant m\leqslant t/2$, $t\leqslant k\leqslant d$ and $\mu$ is a weighting on $Q_d$ with 
\begin{equation}\label{bound}
w_I^m\geqslant\frac{1}{2}\left(1+\frac{\sum_{i=0}^m\binom{\left\lfloor\frac{k-1}{2}\right\rfloor}{t-i}\binom{\lceil\frac{k+1}{2}\rceil}{i}}{\binom{k}{t}}\right)
\end{equation}
for all $I\subseteq[d]$ with $|I|=t$, then there is no $X\in Q_d$ with $|X|=k$ and $P_1(X,\0)>\frac{1}{2}$.
\end{theorem}

The conclusion of this theorem says that neither player can improve their payoff from $(\0,\0)$ by moving to a position at distance $k$ from $\0$. If we can show this for all $k\in[d]$ then $(\0,\0)$ much be an equilibrium. Putting $t=1,m=0$ and noting that the strongest condition is when $k=d$ (if $d$ is even) or $k=d-1$ (when $d$ is odd) we recover Theorem \ref{hypercube-mvt1}. However, the conditions given are genuinely stronger. For example, consider the distribution on $Q_3$ give by
\[
\mu(0,0,1)=\mu(0,1,0)=\mu(1,0,0)=0.3, \quad \mu(1,1,1)=0.1, \quad \mu(0,0,0)=\mu(1,1,0)=\mu(1,0,1)=\mu(0,1,1)=0.
\]
We have $w_i^0=0.6<\frac{3}{4}-\frac{1}{12}=\frac{2}{3}$ for all $i$ and so Theorem \ref{hypercube-mvt1} gives no information about equilibria. However, consider the $t=2,m=1$ case which looks at the number of voters preferring $1$ on at least one of a given pair of issue. We see that $w_I^1=0.9$ for all pairs $I$ while the righthand side of inequality (\ref{bound}) in Theorem \ref{hypercube-mvt2} is $\frac{1}{2}$ when $k=2$ and $\frac{1}{2}$ when $k=3$. Clearly, no player can improve by moving distance $1$ from $(\0,\0)$ (by the definition of the majority point). Together we see that the stronger conditions of Theorem \ref{hypercube-mvt2} are enough to show that $(\0,\0)$ is an equilibrium in this small example.

\begin{proof}[Proof of Theorem \ref{hypercube-mvt2}]
We will prove this via the contrapositive. Take any $0\leqslant m\leqslant t/2$, $t\leqslant k\leqslant d$ and suppose that $X\in Q_d$ with $|X|=k$ and $P_1(X,\0)>1/2$.

Similarly to the proof of Theorem \ref{hypercube-mvt1} consider the sum
\[
S=\sum_{I\subseteq[k], |I|=t} w_I^m
\]
As in Theorem \ref{hypercube-mvt1}, we will count the contribution of each $V\in Q_d$ to $S$ by considering how many choices for $I$ give a summand $w_I^m$ which $\mu(V)$ is included in. If $V\in V(\0,X)$ then $V$ may contribute $\binom{k}{t}\mu(V)$ to $S$ (in the case that $V$ is $0$ on at most $m$ coordinates in $[k]$). 

If $V\in V(X,\0)$ then $|\{i : 1\leqslant i\leqslant k, v_i=0\}|\leqslant\frac{k-1}{2}$ and so the contribution of $V$ to $S$ is at most
\[
\mu(V)\sum_{i=0}^m\binom{\lfloor\frac{k-1}{2}\rfloor}{t-i}\binom{\lceil\frac{k+1}{2}\rceil}{i}.
\]

Finally, if $V\in T(X,\0)$ then $k$ must be even and $|\{i : 1\leqslant i\leqslant k, v_i=0\}|=k/2$ and so the contribution of $V$ to $S$ is exactly
\[
\mu(V)\sum_{i=0}^m\binom{\frac{k}{2}}{t-i}\binom{\frac{k}{2}}{i}.
\]

Aggregating the bounds on all these contributions and setting $\mu(V(\0,X))=x_1$, $\mu(T(X,\0))=x_2$, $\mu(V(X,\0))=x_3$ where $x_1+x_2+x_3=1$ we get:
\begin{align*}
S&\leqslant\binom{k}{t} x_1 +\sum_{i=0}^m\binom{\frac{k}{2}}{t-i}\binom{\frac{k}{2}}{i} x_2 + \sum_{i=0}^m\binom{\lfloor\frac{k-1}{2}\rfloor}{t-i}\binom{\lceil\frac{k+1}{2}\rceil}{i} x_3
\end{align*}
We know that $x_1+\frac{1}{2}x_2=P_2(X,\0)<\frac{1}{2}$. The condition $m<\frac{t}{2}$ means that the coefficient of $x_2$ is less than $\frac{1}{2}\binom{k}{t}$ and so the righthand side is maximised (subject to $x_1+\frac{1}{2}x_2<\frac{1}{2}$) when $x_2=0$. We conclude that
\[
S<\frac{1}{2}\binom{k}{t}+\frac{1}{2}\sum_{i=0}^m\binom{\lfloor\frac{k-1}{2}\rfloor}{t-i}\binom{\lceil\frac{k+1}{2}\rceil}{i}
\]

It follows by averaging, there must be some $I\subseteq[k]$ with $|I|=t$ and
\[
w_I^m<\frac{1}{2}+\frac{1}{2}\frac{\sum_{i=0}^m\binom{\lfloor\frac{k-1}{2}\rfloor}{t-i}\binom{\lceil\frac{k+1}{2}\rceil}{i}}{\binom{k}{t}}.
\]
\end{proof}

The next result shows that Theorem \ref{hypercube-mvt1} is best possible.  

\begin{theorem}\label{non-equilib}
\begin{enumerate}
\item For any odd $d$ and $c<\frac{3}{4}-\frac{1}{4d}$, there is a distribution on $Q_d$ with $w_i^0=c$ for all $i$ which has no equilibrium.
\item For any even $d$ and $c<\frac{3}{4}-\frac{1}{4(d-1)}$, there is a distribution on $Q_d$ with $w_i^0=c$ for all $i$ which has no equilibrium.
\end{enumerate}
\end{theorem}

\begin{proof}
Suppose that $d$ is odd. We construct such a distribution by choosing $\epsilon>0$ and setting
\[
\mu(X)=\begin{cases}
\frac{1}{2}-\epsilon & \text{ if } X=\0\\
(\frac{1}{2}+\epsilon)\binom{d}{\frac{d+1}{2}}^{-1} & \text{ if } |X|=\frac{d+1}{2}\\
0 & \text{ otherwise}
\end{cases}
\]
We have that $P_1(\1,\0)=\frac{1}{2}+\epsilon>\frac{1}{2}$ and so $(\0,\0)$ is not an equilibrium. So by Theorem \ref{weighted-equilibrium} this distribution has no equilibrium.

Since the distribution is symmetric in the coordinates, $w_i^0$ does not depend on $i$ and we have
\begin{align*}
w_i^0&=\left(\frac{1}{2}-\epsilon\right)+\binom{d-1}{\frac{d+1}{2}}\binom{d}{\frac{d+1}{2}}^{-1}\left(\frac{1}{2}+\epsilon\right)\\
&=\left(\frac{1}{2}-\epsilon\right)+\frac{d-1}{2d}\left(\frac{1}{2}+\epsilon\right)\\
&=\frac{3}{4}-\frac{1}{4d}-\epsilon\left(\frac{1}{2}+\frac{1}{2d}\right)  
\end{align*}
which can be made arbitrarily close to $\frac{3}{4}-\frac{1}{4d}$ as required. 

If $d$ is even we use the same construction in the $(d-1)$-dimensional subcube obtained by setting the final coordinate to be $0$. More precisely define:
\[
\mu(X)=\begin{cases}
\frac{1}{2}-\epsilon & \text{ if } X=\0\\
(\frac{1}{2}+\epsilon)\binom{d-1}{\frac{d}{2}}^{-1} & \text{ if $X_d=0$,} |X|=\frac{d}{2}\\
0 & \text{ otherwise}
\end{cases}
\]

In this distribution we have $P_1([d-1],\0)>\frac{1}{2}$ and so again $(\0,\0)$ is not an equilibrium.
\end{proof}

A similar construction demonstrates that the bound of Theorem \ref{hypercube-mvt2} is also optimal; we omit the details of the calculation.

\section{The Balanced $w_i^0=\frac{1}{2}$ case}

In this section we briefly discuss the consequences of our assumption that $w_i^0\not=\frac{1}{2}$ for all $i$. Suppose that this fails and (without loss of generality) we have $w_i^0=\frac{1}{2}$ for $1\leqslant i\leqslant m$ and $w_i^0>\frac{1}{2}$ for $m<i\leqslant d$. Rather than a majority point, we have a \emph{majority subcube} defined by $\{X\in Q_d : X_{m+1}=\ldots=X_d=0\}$.

The argument of Proposition \ref{weighted-equilibrium} implies that if $(A,B)$ is an equilibrium then both $A$ and $B$ must lie in this majority subcube. 

We construct examples which illustrate several possible behaviours. For simplicity, these all have $w_i^0=\frac{1}{2}$ for all $i\in[d]$ (the majority subcube is the whole hypercube). Together these examples show that in this balanced case it is possible to have all pairs, some pairs, or no pairs as equilibria

\begin{example}
If $\mu$ is the uniform measure (that is $\mu(X)=2^{-d}$ for all $X\in Q_d$) then any pair of points $(A,B)$ is an equilibrium.
\end{example}

\begin{example}
Let $d$ be odd and define $\mu$ according to the parity of $|X|$ by
\[
\mu(X)=\begin{cases}
2^{-d-1}& |X| \text{ is even}\\
0& |X| \text{ is odd}
\end{cases}
\]
It is easy to check that if $A,B\in Q_d$ do not form an antipodal pair then $P_1(A,B)=P_2(A,B)=\frac{1}{2}$. If $|A|$ is even (so $|A^c|$ is odd) then $P_1(A^c,A)>\frac{1}{2}$. It follows that the pair $(A,B)$ is an equilibrium if and only both $A$ and $B$ are odd parity vertices.
\end{example}

\begin{example}
Let $d=5$ and for $X\in Q_5$ define $\mu(X)=w(|X|)$ where
\[
w(k)=
\begin{cases}
\frac{1}{16}(1-\epsilon)+\frac{3}{8}\epsilon & k=0\\
\frac{1}{16}(1-\epsilon) & k=2\\
\frac{1}{16}(1-\epsilon)+\frac{1}{8}\epsilon & k=4\\
0 & \text{ otherwise}
\end{cases}
\]
where $\epsilon>0$ will be chosen later.

We will show that here is no equilibrium in this example. We need to check that for each $X\in Q_5$ there is a response $Y$ with $P_1(X,Y)<\frac{1}{2}$. Considering the possible cases for $X$:
\begin{itemize}
\item If $|X|$ is odd, we have that the payoff to $X$ against an opponent playing $X^c$ is the weight of the vertices adjacent to $X$. This is at most $\frac{5}{16}+c\epsilon$ for some constant $c$ and so if $\epsilon$ is small enough $P_1(X,X^c)<\frac{1}{2}$.

\item If $|X|=0$ then $P_1(X,11110)=w(0)+4w(2)+\frac{1}{2}6w(2)=\frac{1}{2}-\frac{1}{8}\epsilon<\frac{1}{2}$.

\item If $|X|=2$ then $P_1(X,11111)=w(0)+6w(2)=\frac{7}{16}-\frac{1}{16}\epsilon<\frac{1}{2}$.
 
\item If $X=01111$ then $P_1(X,11100)=3w(4)+5w(2)=\frac{1}{2}-\frac{1}{8}<\frac{1}{2}$ (the remaining cases with $|X|=2$ follow by symmetry).
\end{itemize}

Checking the cases of $|X|$ odd more carefully, we see that any $0<\epsilon<\frac{3}{5}$ will provide a specific example of a balanced probability measure with no equilibria.
\end{example}

\section{Mix of Product Measures}

It is perhaps too much to hope for strong results for such a general concept as an arbitrary weight functions. We turn now to considering what can be said about a particular class of distribution. 

Note first that if $\mu$ is monotone in the sense that $\mu(X)\geqslant \mu(Y)$ for all $X\subseteq Y$ then $(\0,\0)$ is always an equilibrium. Indeed, if $B\in Q_d$ is any point other than $\0$ we can match up all points of $Q_d\setminus T(\0,B)$ into pairs of the form $X,Y$ with $X\subseteq Y$, $X_i=Y_i$ for all $i$ for which $B_i=0$, and $d(\0,X)=d(B,Y)$ (effectively this involves finding matchings between layers $t$ and $k-t$ of a $k$-dimensional cube preserving inclusion.) Now, monotonicity shows that $P(\0,B)\geqslant P(B,\0)$ as required. Similarly, if $\mu(X)\leqslant \mu(Y)$ for all $X\subseteq Y$ then $(\1,\1)$ is always an equilibrium. This is effectively equivalent to the observation on single-peaked distributions in \cite{MMW}. 

In particular, the  product distribution $\mu(X)=p^{|X|}(1-p)^{d-|X|}$ (with $p\not=\frac{1}{2}$) has an equilibrium. In fact, this holds even for the more general product weight
\[
\mu(X)=\prod_{i : x_i=1} p_i \prod_{i : x_i=1} (1-p_i)
\]
(with $p_i\not=\frac{1}{2}$ for all $i$).

Perhaps the next natural case to consider is a mixture of two product distributions. Fix $0<\alpha<1$, $0<p_1<\frac{1}{2}<p_2<1$. We define
\[
\mu(X;d,\alpha,p_1,p_2)=(1-\alpha)p_1^{|X|}(1-p_1)^{d-|X|}+ \alpha p_2^{|X|}(1-p_2)^{d-|X|}.
\]
This distribution has a nice social interpretation; it models a society in which proportion $1-\alpha$ of voters are $0$-leaning and choose option $1$ on each issue with probability $p_1<\frac{1}{2}$, while the remaining proportion $\alpha$ of voters are $1$-leaning and choose option $1$ on each issue with probability $p_2>\frac{1}{2}$. Note that if we had $0<p_1,p_2<\frac{1}{2}$ or $\frac{1}{2}<p_1,p_2<1$ then the resulting distribution is monotone and so we must have an equilibrium by the argument above.

We can have non-equilibrium behaviour here, similar to the construction in Proposition \ref{non-equilib}. For example, if $\alpha=\frac{2}{3}$, $p_1=\frac{1}{5}$, $p_2=\frac{3}{5}$ then (thinking of $w_i^0$ as the probability that a randomly chosen voter prefers $0$ on issue $i$) we see that:
\[
w_i^0=(1-\alpha)(1-p_1)+\alpha(1-p_2)=\frac{8}{15}>\frac{1}{2}
\]
So $M=\0$ but almost all of the total $\frac{2}{3}$ weight coming from the $1$-leaning voters is concentrated close to layer $p_2d$ and so for large $d$ this weight makes $P(\1,\0)>\frac{1}{2}$. The next result shows that when $d$ is large, the existence of an equilibrium for this class of distributions depends on $\alpha$ rather than $p_1$ and $p_2$.

\begin{theorem}\label{mix-prod}
If $0<\alpha<1$, $0<p_1<\frac{1}{2}<p_2<1$ and $\mu(X; \alpha,p_1,p_2)$ is as above and $w_i^0>\frac{1}{2}$ (so that $M=\0$) then one of the following holds:
\begin{itemize}
\item $\0$ is the best response to $\0$ and so $(\0,\0)$ is an equilibrium.
\item $\1$ is the best response to $\0$.
\end{itemize}
Moreover, the second possibility occurs for all large $d$ if and only if $\alpha>\frac{1}{2}$. 
\end{theorem}
Of course, if $w_i^0<\frac{1}{2}$ then a similar statement applies with $M=\1$.

Notice that by choosing $\alpha=\frac{1}{2}+\epsilon$, $p_1=\frac{1}{2}-\epsilon$, $p_2=\frac{1}{2}+\epsilon$ and $d$ large we get an example with no equilibrium in which $w_i^0=(1-\alpha)(1-p_1)+\alpha(1-p_2)$ can be made arbitrarily close to $\frac{3}{4}$, asymptotically matching the bound in Theorem \ref{hypercube-mvt1}.

Before proving this, we make an observation which will be useful when calculating payoffs in this mixed product distribution.

If $I\subseteq[d]$, we write $Q_I$ for the set of all $d$-tuples of $0$s, $1$s and $*$s which take value $*$ in position $i$ if and only if $i\not\in I$. 
\[
Q_I=\{X\in\{0,1,*\}^d : x_i=* \text{ if and only if } i\not\in I\}
\]
As before, for $X\in Q_I$ we write $|X|$ for the number of $1$s in $X$.

If $X,Y\in Q_d$, then to decide which of $V(X,Y), V(Y,X), T(X,Y)$ a point $V$ lies in, we only need to look at the coordinates $i$ for which $x_i\not= y_i$. We denote the set of these coordinates by $X\triangle Y$. Now, given a distribution $\mu$ on $Q_d$ we define a new distribution $\mu_{X,Y}:Q_{X\triangle Y} \rightarrow \mathbb{R}_{\geqslant 0}$ on the subcube generated by $X$ and $Y$. This is defined by
where
\[
\mu_{X,Y}(U)=\mu(\{V\in Q_d : v_i=u_i \text{ for all } i\in X\triangle Y\}.
\]
The key point is that to determine $P(X,Y)$ we only need to look at $\mu_{X,Y}$. The cleanest to state and most useful application of this is the case $X=\0$.
\begin{prop}\label{subcube-prop}
For any distribution $\mu$ and any $Y\in Q_d$, if we define a new distribution $\mu_{\0,Y}$ as above then: 
\[
P(\0,Y)=\mu_{\0,Y}(\{V\in Q_Y : |Y|<|Y|/2\})+\frac{1}{2}\mu_{\0,Y}(\{V\in Q_Y : |Y|=|Y|/2\})
\]
\end{prop}
In other words we can express the payoff to a player choosing $Y$ against $\0$ in our original weighting $\mu$ as the weight of the `top half of layers' of a re-weighted $|Y|$-dimensional subcube. 

\begin{proof}[Proof of Theorem \ref{mix-prod}]
Let $a_k$ be the payoff to the the $\1$ player in the state $(\1,\0)$ under the distribution $\mu(\cdot; \alpha,p_1,p_2)$ on $Q_k$. It is also the case that in any $d\geqslant k$ we have $P_1(X,\0)=a_k$ for any $X$ with $|X|=k$. This follows from Proposition \ref{subcube-prop} and the fact that for the product measure (and hence for any convex combination of product measures) the restriction $\mu_{\0,X}$ is itself a product measure on $Q_X$.   
 
It suffices then to analyse carefully how $a_d$ grows with $d$. Let $0<p<1$, $q=1-p$ and define
\[
x_d(p)=\sum_{i=\lceil\frac{d+1}{2}\rceil}^d \binom{d}{i}p^iq^{d-i}+\frac{1}{2}\binom{d}{\frac{d}{2}}p^{d/2}q^{d/2}
\]
(ie the payoff $P_1([d],\0)$ in $Q_d$ with the product measure).

If $d$ is odd then:
\[
x_{d+1}(p)-x_d(p)=\frac{1}{2}\binom{d}{\frac{d-1}{2}}p^{\frac{d-1}{2}}q^{\frac{d+1}{2}}p-\frac{1}{2}\binom{d}{\frac{d+1}{2}}p^{\frac{d+1}{2}}q^{\frac{d-1}{2}}q=0
\]
While if $d$ is even then
\[
x_{d+1}(p)-x_d(p)=\frac{1}{2}\binom{d}{\frac{d}{2}}p^{\frac{d}{2}}q^{\frac{d}{2}}p-\frac{1}{2}\binom{d}{\frac{d}{2}}p^{\frac{d}{2}}q^{\frac{d}{2}}q=\frac{1}{2}\binom{d}{\frac{d}{2}}(pq)^{\frac{d}{2}}(p-q)
\]
So, for $p>\frac{1}{2}$ we have that $x_{2i-1}=x_{2i}$ and otherwise $x_d$ is increasing. That is:
\[
\frac{1}{2}=x_0<x_1=x_2<x_3=x_4<\dots.
\]
While if $p<\frac{1}{2}$ we have:
\[
\frac{1}{2}=x_0>x_1=x_2>x_3=x_4>\dots.
\]
It is also easy to see that if $p>\frac{1}{2}$ then $x_d\to 1$ while if $p<\frac{1}{2}$ then $x_d\rightarrow 0$.

Considering how the mixed product weight is defined, we have
\[
a_d=(1-\alpha)x_d(p_1)+\alpha x_d(p_2).
\]
(the sum of an increasing term and a decreasing term). So
\begin{align*}
a_{d+1}-a_d&=(1-\alpha)\left(x_{d+1}(p_1)-x_d(p_1)\right)+\alpha\left(x_{d+1}(p_2)-x_d(p_2)\right)\\
&=\left\{\begin{array}{ll}
\frac{1}{2}\binom{t}{\frac{t}{2}}\left((1-\alpha)(p_1q_1)^{\frac{t}{2}}(p_1-q_1)+\alpha(p_2q_2)^{\frac{t}{2}}(p_2-q_2)\right) & \text{ if $d$ is even}\\
0 & \text{ if $d$ is odd}
\end{array}\right.
\end{align*} 

We deduce that $a_{d+1}>a_d$ for even $d$ if and only if:
\[
\frac{\alpha}{(1-\alpha)}\left(\frac{p_2q_2}{p_1q_1}\right)^{\frac{d}{2}}\frac{p_2-q_2}{q_1-p_1}>1.
\]
Suppose now that $w_i^0=(1-\alpha)q_1+\alpha q_2>\frac{1}{2}$. We have $a_0=\frac{1}{2}, a_1=1-w_i^0$ so the sequence $a_k$ is initially decreasing. If $\left(\frac{p_2q_2}{p_1q_1}\right)\leqslant 1$ then the sequence will continue to decrease. If $\left(\frac{p_2q_2}{p_1q_1}\right)>1$ then it will decrease to a minimum and then increase. In either case, the maximum value of $a_k$ is attained either at $k=0$ or $k=d$. In the former case we have that $\0$ is a best response to $\0$ and so $(\0,\0)$ is an equilbirium. In the later case, if additionally $a_d>a_0=\frac{1}{2}$, we have that $\1$ is a best response to $\0$.

Finally, notice that as $d\to\infty$, $a_d\to\alpha$. It follows that the second (non-equilibrium) case occurs for large $d$ if and only if $\alpha>\frac{1}{2}$.
\end{proof}

\section{Further Work}

There are many open questions about this model or its variants. We mention three possible directions for future work.

The first is to find conditions guaranteeing an equilibrium based on something other than the quantities $w_I^m$. This could involve results about general structural conditions such as in Theorems \ref{hypercube-mvt1} and \ref{hypercube-mvt2} or particular classes of distribution such as in Theorem \ref{mix-prod}.

If there is no equilibrium, then a dynamic process arises as players move in turn to improve their payoff. The details of this depend on the exact moving rule; two natural examples would be to move to the position which maximises payoff given the other players position, or to make the shortest distance move which increases your payoff. Under each of these, what kind of trajectories do the players follow? This questions could be asked of the mix of product measures considered in Section 5 or more generally.

Finally, Theorem \ref{mix-prod} shows that there is a natural class of distributions in which the best response to the majority point is its antipode. Could this point towards a strategic explanation of polarisation. For instance, are there general conditions on the distribution under which an antipodal pair of positions forms a $k$-local equilibrium for some moderately large $k$?


\begin{thebibliography}{1}
\bibitem{downs} A. Downs. An economic theory of democracy. Harper \& Row New York, 1957.
\bibitem{DT} C. D\"urr, N.K. Thang, Nash Equilibria in Voronoi Games on Graphs, In: {\em Proceedings of the 15th Annual European Symposium on Algorithms (ESA)}, pages 17--28, 2007.
\bibitem{FMM} R. Feldmann, M. Mavronicolas and B. Monien. Nash Equilibria for Voronoi Games on Transitive Graphs. In: {\em Proceedings of the 5th International Workshop on Internet and Network Economics (WINE)}, 280--291, 2009.
\bibitem{GMPPR} D. Gerbner, V. M\'esz\'aros, D. P\'alv\"olgyi, A. Pokrovskiy and G. Rote. Advantage in the discrete Voronoi game. {\em J. Graph Algorithms Appl.} \textbf{18}, 439--457, 2014.
\bibitem{hotelling} H. Hotelling. Stability in competition. {\em Economic Journal} \textbf{39}, 41--57, 1929.
\bibitem{MMW} J. Maass, V. Mousseau and A. Wilczynski. A Hotelling-Downs Game for Strategic Candidacy with Binary Issues. {\em Proc. 2023 International Conference on Autonomous Agents and Multiagent Systems}, 2076--2084, 2023.
\bibitem{MMPS} M. Mavronicolas, B. Monien, V.G. Papadopoulou and F. Schoppmann. Voronoi Games on Cycle Graphs. In: {\em Proceedings of the 33rd International Symposium on Mathematical Foundations of Computer Science (MFCS)}, 503--514, 2008.
\bibitem{TDU} S. Teramoto, E.D. Demaine and R. Uehara. The Voronoi game on graphs and its complexity. {\em J. Graph Algorithms Appl.}, \textbf{15}, 485--501, 2011.

\end{thebibliography}
\end{document}